\documentclass[12pt]{article}
\usepackage{amsmath}
\usepackage{graphicx}
\usepackage{enumerate}
\usepackage{natbib}
\usepackage{url} 

\usepackage[utf8]{inputenc} 
\usepackage[T1]{fontenc}    
\usepackage{hyperref}       
\usepackage{url}            
\usepackage{booktabs}       
\usepackage{tabularx}       
\usepackage{amsfonts}       
\usepackage{amsmath,mathtools,amsthm}	
\usepackage{amssymb,dsfont,bbm}	
\usepackage{nicefrac}       
\usepackage{microtype}      
\usepackage{xcolor}         
\usepackage{tikz-cd}        

\usepackage{xcolor}
\usepackage{srcltx}
\usepackage{etoolbox}
\usepackage{nameref}
\usepackage{comment}
\usepackage{mathrsfs}
\usepackage{enumerate}
\usepackage[shortlabels]{enumitem}
\usepackage{amsfonts} 
\usepackage{nicefrac} 
\usepackage{mathtools}
\usepackage{dsfont,mathrsfs}
\usepackage{cleveref}
\usepackage{xargs}
\usepackage{multirow}
\usepackage{bm}
\usepackage{subfigure}

\usepackage{autonum}
\usepackage{upgreek}

\newtheorem{assum}{A\hspace{-2pt}}

\newtheorem{theorem}{Theorem}
\crefname{theorem}{theorem}{Theorems}
\Crefname{theorem}{Theorem}{Theorems}

\newtheorem{lemma}{Lemma}
\crefname{lemma}{lemma}{lemmas}
\Crefname{lemma}{Lemma}{Lemmas}

\crefname{remark}{remark}{remarks}
\Crefname{remark}{Remark}{Remarks}

\crefname{corollary}{corollary}{corollaries}
\Crefname{corollary}{Corollary}{Corollaries}

\newtheorem{proposition}{Proposition}
\crefname{proposition}{proposition}{propositions}
\Crefname{proposition}{Proposition}{Propositions}

\crefname{definition}{definition}{definitions}
\Crefname{Definition}{Definition}{Definitions}

\crefname{example}{example}{examples}
\Crefname{Example}{Example}{Examples}

\crefname{figure}{figure}{figures}
\Crefname{Figure}{Figure}{Figures}

\crefname{table}{table}{tables}
\Crefname{Table}{Table}{Tables}

\crefname{assum}{A\hspace{-2pt}}{A\hspace{-2pt}}
\crefname{assumb}{B\hspace{-2pt}}{B\hspace{-2pt}}
\crefname{assumUGE}{UGE\hspace{-1pt}}{UGE\hspace{-1pt}}
\crefname{assumID}{IND\hspace{-1pt}}{IND\hspace{-1pt}}
\crefname{assumUE}{UE\hspace{-1pt}}{UE\hspace{-1pt}}
\crefname{assumSUP}{M\hspace{-1pt}}{M\hspace{-1pt}}

\newlist{renumerate}{enumerate}{3}
\setlist[renumerate]{wide, labelwidth=!, labelindent=0pt,label=(\roman*)}

\newlist{aenumerate}{enumerate}{3}
\setlist[aenumerate]{wide, labelwidth=!, labelindent=0pt,label=(\arabic*)}

\newlist{aaenumerate}{enumerate}{3}
\setlist[aaenumerate]{wide, labelwidth=!, labelindent=0pt,label=(\alph*)}

\newlist{aenumerateSpace}{enumerate}{3}
\setlist[aenumerateSpace]{wide, labelwidth=!,label=(\arabic*)}

\newlist{benumerate}{enumerate}{3}
\setlist[benumerate]{wide, labelwidth=!, labelindent=0pt,label=$\bullet$}

\newcommand{\PE}{\mathbb{E}}

\newcommand{\PP}{\mathbb{P}}
\newcommandx{\genericb}[1][1=]{b_{#1}}
\newcommandx{\Constros}[1][1=]{\operatorname{C}_{\operatorname{Ros},#1}}
\newcommandx{\Constburk}[1][1=]{\operatorname{C}_{\operatorname{Burk}}}
\newcommandx{\driftW}[1][1=]{W_{#1}}

\newcommandx{\metricd}[1][1=]{\mathsf{d}_{#1}}

\newcommandx\invmeasure[1][1=]{\Pi_{#1}}
\newcommandx{\PPjoint}[1][1=]{\PP^{\MKjoint[#1]}}
\newcommandx{\PEjoint}[1][1=]{\PE^{\MKjoint[#1]}}
\newcommandx{\PEMID}[1][1=\alpha]{\PE^{\MK[#1]}}
\newcommandx{\PPMID}[1][1=\alpha]{\PP^{\MK[#1]}}

\newcommandx{\MKjoint}[1][1=]{\bar{\operatorname{P}}_{#1}}
\newcommandx\costw[1][1=]{\mathsf{c}_{#1}}

\newcommandx\Intergrdist[1][1=]{\mathbb{M}_{1}(#1)}

\newcommandx{\mmarkov}[1][1=0]{m^{(\Markov)}_{#1}}

\def\F{\mathcal{F}}
\def\RemU{\bar{\mathcal{R}}}

\def\Zset{\mathsf{Z}}
\def\Zsigma{\mathcal{Z}}

\def\rset{\mathbb{R}}

\def\nset{\ensuremath{\mathbb{N}}}
\def\nsets{\ensuremath{\mathbb{N}^*}}

\newcommandx\sequence[4][2=,3=,4=]
{\ifthenelse{\equal{#3}{}}{\ensuremath{\{ #1_{#2 #4}\}}}{\ensuremath{\{ #1_{#2 #4} \}_{#2 \in #3}}}}

\newcommandx\sequenceD[2][2=]
{\ifthenelse{\equal{#2}{}}{\ensuremath{\{ #1\}}}{\ensuremath{\{ #1\!~:\!~#2  \}}}}
\newcommandx\sequenceDouble[4][3=,4=]
{\ifthenelse{\equal{#3}{}}{\ensuremath{\{ (#1_{#3},#2_{#3})\}}}{\ensuremath{\{ (#1_{#3},#2_{#3})\}_{#3 \in #4}}}}

\newcommandx{\sequencen}[2][2=n\in\nset]{\ensuremath{\{ #1, \eqsp #2 \}}}
\newcommandx\sequencens[2][2=n]
{\ensuremath{\{ #1_{#2} \!~:\!~#2\in\nsets\}}}
\newcommandx\sequencet[4]
{\ensuremath{\{ #1{#2_{#3}} \, : \, \eqsp #3 \in #4 \}}}
\def\PE{\mathbb{E}}

\newcommandx{\PVar}[1][1=]{\ensuremath{\operatorname{Var}_{#1}}}

\newcommandx{\MK}[1][1=\alpha]{\mathrm{P}_{#1}}
\newcommandx\MKK[1][1=\alpha]{\mathrm{K}_{#1}}
\def\MKQ{\mathrm{P}}
\def\MKR{\mathrm{R}}

\newcommandx{\PEtilde}[1][1=]{\PE^{\mathrm{K}_{#1}}}
\newcommandx{\PPtilde}[1][1=]{\PP^{\mathrm{K}_{#1}}}

\newcommandx{\norm}[2][2=]{\Vert#1 \Vert_{{#2}}}
\newcommandx{\normLigne}[2][2=]{\Vert#1 \Vert_{{#2}}}
\newcommandx{\normLine}[2][2=]{\Vert#1 \Vert_{{#2}}}
\newcommandx{\normop}[2][2=]{\Vert{#1}\Vert_{{#2}}}
\newcommandx{\normopLigne}[2][2=]{\Vert{#1}\Vert_{{#2}}}
\newcommandx{\normopLine}[2][2=]{\Vert{#1}\Vert_{{#2}}}
\newcommandx{\osc}[2][1=]{\mathrm{osc}_{#1}(#2)}

\newcommandx{\normlip}[2][2=\operatorname{Lip}]{\Vert#1 \Vert_{{#2}}}
\newcommand{\lip}{\operatorname{L}}
\newcommandx{\lipspace}[1]{\lip_{#1}}

\newcommandx{\CPP}[3][1=]
{\ifthenelse{\equal{#1}{}}{{\mathbb P}\left(\left. #2 \, \right| #3 \right)}{{\mathbb P}_{#1}\left(\left. #2 \, \right | #3 \right)}}
\newcommandx{\CPPtilde}[3][1=]
{\ifthenelse{\equal{#1}{}}{{\tilde{\mathbb P}}\left(\left. #2 \, \right| #3 \right)}{{\tilde{\mathbb P}}_{#1}\left(\left. #2 \, \right | #3 \right)}}

\newcommandx{\as}[1][1=\PP]{\ensuremath{#1\, -\mathrm{a.s.}}}

\newcommand{\eqsp}{\;}

\newcommandx{\boundmetric}[1][1=]{\kappa_{\MKK[#1]}}

\newcommandx{\Nnorm}[2][1=V]{[ #2]_{#1}}
\newcommandx{\lipnorm}[2][1=g]{[ #1]_{#2}}

\newcommandx{\CPE}[3][1=]{{\mathbb E}^{#3}_{#1}\left[#2\right]}
\newcommandx{\CPEext}[3][1=]{\tilde{\mathbb E}^{#3}_{#1}\left[#2\right]}
\newcommandx{\CPEtilde}[3][1=]{{\tilde{\mathbb E}}^{#3}_{#1}\left[#2\right]}
\newcommandx{\CPEs}[3][1=]{{\mathbb E}^{#3}_{#1}[#2]}

\def\trace{\operatorname{Tr}}

\newcommand{\rmd}{\mathrm{d}}
\def\funcAw{\mathbf{A}}

\def\funcbw{\mathbf{b}}

\newcommandx{\zmfuncA}[2][1=]{\tilde{\funcAw}^{#1}(#2)}
\newcommandx{\zmfuncAw}[1][1=]{\tilde{\funcAw}_{#1}}
\newcommandx{\zmfuncb}[2][1=]{\tilde{\funcbw}^{#1}(#2)}

\newcommandx{\funcct}[2][1=]{\funcctilde^{#1}(#2)}

\def\taumix{t_{\operatorname{mix}}}

\newcommand{\1}{\boldsymbol{1}}

\newcommandx{\CovC}[1][1=u]{\operatorname{C}_{#1}}

\usepackage[textwidth=2.25cm,textsize=footnotesize]{todonotes}

\DeclareMathAlphabet{\mathpzc}{OT1}{pzc}{m}{it}

\def\lyapW{\mathpzc{W}}

\newcommandx{\bias}[1][1=\alpha]{\operatorname{B}_{#1}}

\newcommandx\probaMarkovTilde[2][2=]
{\ifthenelse{\equal{#2}{}}{{\widetilde{\mathbb{P}}_{#1}}}{\widetilde{\mathbb{P}}_{#1}\left[ #2\right]}}

\def\funcctilde{\tilde{c}_u}

\def\MKR{\mathsf{R}}

\newcommandx{\driftb}[1][1=p]{\bar{b}_{#1}}

\def\tvdist{\mathsf{d}_{\operatorname{tv}}}

\newcommandx{\boldb}[1][1={q}]{\mathsf{b}_{#1}}

\newcommandx{\ConstGW}[1][1={n,\lyapW}]{\operatorname{G}_{#1}}

\newcommandx{\ConstMW}[1][1={n,\lyapW}]{\operatorname{M}_{#1}}

\Crefname{assumTD}{\textbf{TD}\hspace{-1pt}}{\textbf{TD}\hspace{-1pt}}
\crefname{assumTD}{\textbf{TD}}{\textbf{TD}}

\Crefname{assumptionC}{\textbf{C}\hspace{-1pt}}{\textbf{C}\hspace{-1pt}}
\crefname{assumptionC}{\textbf{C}}{\textbf{C}}

\Crefname{assumptionM}{\textbf{UGE}\hspace{-1pt}}{\textbf{UGE}\hspace{-1pt}}
\crefname{assumptionM}{\textbf{UGE}}{\textbf{UGE}}

\def\distance{\mathsf{d}}

\newcommandx{\vartconstwas}[1][1=V]{c_{#1}}

\newcommandx{\deltawas}[1][1=*]{\delta_{#1}}

\newcommandx{\wasser}[4][1=\distance,4=]{\mathbf{W}_{#1}^{#4}\left(#2,#3\right)}
\newcommandx{\covcoeff}[2]{\rho_{#1}^{(#2)}}

\newcommand{\dobrush}{\mathsf{\Delta}}
\newcommandx{\dobru}[3][1=,3=]{\dobrush_{#1}^{#3}( #2)}  

\def\Markov{\mathrm{M}}

\newcommandx{\dlim}[1]{\ensuremath{\stackrel{#1}{\Longrightarrow}}}

\def\OBM{\mathsf{OBM}}

\title{A note on concentration inequalities for the overlapped batch mean variance estimators for Markov chains}
\author{Eric Moulines~\footnote{Ecole Polytechnique, France, and MBUZAI, UAE, \texttt{eric.moulines@polytechnique.edu}.}, Alexey Naumov~\footnote{HSE University, Russia, \texttt{anaumov@hse.ru}.}, and Sergey Samsonov~\footnote{HSE University, Russia, \texttt{svsamsonov@hse.ru}.}}

\begin{document}

\maketitle

\begin{abstract}
In this paper, we study the concentration properties of quadratic forms associated with Markov chains using the martingale decomposition method introduced by Atchadé and Cattaneo (2014). In particular, we derive concentration inequalities for the overlapped batch mean (OBM) estimators of the asymptotic variance for uniformly geometrically ergodic Markov chains. Our main result provides an explicit control of the $p$-th moment of the difference between the OBM estimator and the asymptotic variance of the Markov chain with explicit dependence upon $p$ and mixing time of the underlying Markov chain.
\end{abstract}

\section{Introduction}
\label{sec: intro}
We consider a Markov kernel $\MKQ$ on a measurable space $(\Zset, \Zsigma)$, assuming that it admits a unique stationary distribution $\pi$. 
Let $(Z_k)_{k \in \nset}$ denote a Markov chain with kernel $\MKQ$ and initial distribution $\xi$. Without loss of generality, we assume $(Z_k)_{k \in \nset}$ to be a canonical chain defined on the respective canonical space. Given a measurable function $f: \Zset \to \rset$ satisfying $\pi(f) = 0$, we consider the quadratic form
\begin{equation}
\label{eq:u_stat_general_def}
U_n(f) = \sum_{\ell=1}^{n}\sum_{j=1}^{\ell} w(\ell,j) f(Z_{\ell}) f(Z_j),
\end{equation}
with the weight matrix $W_n = (w(\ell,j))_{\ell,j=1}^n$. Using the martingale decomposition method introduced by \cite{atchade2014martingale}, we derive in this paper explicit moment bounds for the quadratic form $U_n(f) - \PE[U_n(f)]$. In particular, our analysis provides the $p$-th moment bounds with explicit dependence on the parameters of the Markov kernel $\MKQ$, function $f$, weight matrix $W$ and the moment order $p$. We apply these bounds to study the concentration properties of the overlapped batch mean (OBM) estimator for asymptotic variance of Markov chains in \Cref{sec:variance}.

\paragraph{Problem setting.} We consider the setting of a Markov kernel $\MKQ$ which is uniformly geometrically ergodic. This assumption allows us to present the core elements of our methodology clearly and without excessive technical complexity. Precisely, we formulate the following assumption:

\begin{assum}
\label{assum:UGE}
The sequence $(Z_k)_{k \in \nset}$ is a Markov chain taking values in $(\Zset,\Zsigma)$ with the Markov kernel $\MKQ$. Moreover, the Markov kernel $\MKQ$ admits $\pi$ as a unique invariant distribution and there exists $\taumix \in \nset$, such that for any $k \in \nset$, it holds that 
\begin{equation}
\label{eq:tau_mix_contraction}
\sup_{z,z' \in \Zset} \tvdist(\MKQ^{k}(z,\cdot),\MKQ^{k}(z',\cdot)) \leq (1/4)^{\lceil k / \taumix \rceil}\eqsp.
\end{equation}
\end{assum}
Parameter $\taumix$ in \eqref{eq:tau_mix_contraction} is referred to as \emph{mixing time}, see e.g. \cite{paulin_concentration_spectral}. We define the function $g(z): \Zset \to \rset$ as a solution to the Poisson equation, associated to the function $f$, that is, 
\begin{equation}
\label{eq:Pois_eq_with_f}
g(z) - \MKQ g(z) = f(z)\eqsp.
\end{equation}
Under the UGE assumption \Cref{assum:UGE}, there exists a unique solution to the above equation \eqref{eq:Pois_eq_with_f} (see e.g. \cite[Chapter~21]{douc:moulines:priouret:soulier:2018}). Moreover, this solution is given by the formula 
\[
g(z) = \sum_{k=0}^{\infty}\{\MKQ^{k}f(z) - \pi(f)\} = \sum_{k=0}^{\infty} \MKQ^{k}f(z)\eqsp.
\]
For $\ell \in \{1,\ldots,n\}$ we define
\begin{equation}
\label{eq:martingale_increment_def}
M_{\ell} = g(Z_{\ell}) - \MKQ g(Z_{\ell-1})\eqsp, 
\end{equation}
and for $j \leq \ell$, we define 
\begin{equation}
\label{eq:delta_M_increment_defi}
\Delta M_{\ell,j} = (g(Z_j) - \MKQ g(Z_{j-1}))(g(Z_{\ell}) - \MKQ g(Z_{\ell-1}))\eqsp.
\end{equation}
Set $\F_{\ell} = \sigma(Z_{j}, j \leq \ell)$. Since $\CPE{ M_{\ell}}{\F_{\ell-1}} = 0$, $ M_{\ell}$ is a martingale increment sequence. Moreover, for $\ell > j$, it holds that $\CPE{\Delta M_{\ell,j}}{\F_{\ell-1}} = 0$, and $\Delta M_{\ell,j}$ is for all $\ell > j$ a martingale increment sequence. Note also that 
\[
\Delta M_{\ell,\ell} = (g(Z_{\ell}) - \MKQ g(Z_{\ell-1}))^2\eqsp.
\]
Following \cite{atchade2014martingale}, we obtain the following result:
\begin{proposition}[Lemma~2.2 in \cite{atchade2014martingale}]
\label{prop:quadr_form_decomposition}
Assume \Cref{assum:UGE}. Then for any bounded function $f: \Zset \to \rset$, it holds that 
\begin{align}
\label{eq:key_representation_U_stat}
U_n(f) = \sum_{\ell=1}^{n} w(\ell,\ell) \Delta M_{\ell,\ell} + \sum_{\ell=1}^{n}\sum_{j=1}^{\ell-1}w(\ell,j) \Delta M_{\ell,j}  + \RemU_{n}\eqsp,
\end{align}
where we have defined
\begin{equation}
\label{eq:reminder_u_stat}
\begin{split}
\RemU_{n} 
&= \underbrace{\sum_{\ell=1}^{n}\sum_{j=1}^{\ell} w(\ell,j) g(Z_{j}) \left(\MKQ g(Z_{\ell-1}) - \MKQ g(Z_{\ell})\right)}_{T_1} \\
&+ \underbrace{\sum_{\ell=1}^{n}\sum_{j=1}^{\ell} w(\ell,j) g(Z_{\ell}) \left(\MKQ g(Z_{j-1}) - \MKQ g(Z_{j})\right)}_{T_2} \\
&+ \underbrace{\sum_{\ell=1}^{n}\sum_{j=1}^{\ell} w(\ell,j) \left( \MKQ g(Z_{j}) \MKQ g(Z_{\ell}) - \MKQ g(Z_{j-1}) \MKQ g(Z_{\ell - 1}) \right)}_{T_3}\eqsp.
\end{split}
\end{equation}
\end{proposition}
\begin{proof}
Proof is provided in \Cref{sec:prrof_quadr_form_decomposition}.
\end{proof}

We first provide below an alternative representation for $\RemU_{n}$:
\begin{lemma}
\label{lem:bar_r_n_representation}
The remainder term $\RemU_{n}$ defined in \eqref{eq:reminder_u_stat} can be represented as  
\begin{equation}
\label{eq:bar_r_n_representation}
\begin{split}
\RemU_{n} &= \sum_{\ell=3}^{n} \MKQ g(Z_{\ell-1}) \sum_{j=1}^{\ell-2} \Delta^{(1,0)}(\ell,j) M_{j}  + \sum_{\ell=1}^{n} M_{\ell} \sum_{j=1}^{\ell-1} \Delta^{(0,1)}(\ell,j) \MKQ g(Z_{j-1}) \\
&+ \sum_{\ell=3}^{n-1} \sum_{j=1}^{\ell-2} \Delta^{(1,1)}(\ell,j) \MKQ g(Z_{\ell-1}) \MKQ g(Z_{j-1}) + \RemU_{n,mart} +  \RemU_{n,rem} \eqsp,
\end{split}
\end{equation}
where the coefficients $\Delta^{(1)}(\ell,j)$, $\Delta^{(2)}(\ell,j)$, and $\Delta^{(1,1)}(\ell,j)$ has form 
\begin{equation}
\label{eq:coef_def}
\begin{split}
\Delta^{(1,0)}(\ell,j) &= w(\ell,j) - w(\ell-1,j)\eqsp, \quad \Delta^{(0,1)}(\ell,j) = w(\ell,j) - w(\ell,j-1)\eqsp, \\
\Delta^{(1,1)}(\ell,j) &= w(\ell,j) - w(\ell,j-1) - w(\ell-1,j) + w(\ell-1,j-1)\eqsp, \\
\Delta^{(m)}(\ell,\ell) &= w(\ell,\ell-1) + w(\ell-1,\ell-2) - w(\ell,\ell-2) - 2w(\ell-1,\ell-1)\eqsp,
\end{split}
\end{equation}
and the remainder terms $\RemU_{n,mart}$, $\RemU_{n,rem}$, are defined as 
\begin{equation}
\label{eq:rem_u_mart_def}
\begin{split}
\RemU_{n,mart} &= \underbrace{\sum_{\ell=2}^{n} \Delta^{(1,0)}(\ell,\ell-1) \MKQ g(Z_{\ell-1}) M_{\ell-1} - \sum_{\ell=1}^{n} w(\ell,\ell) M_{\ell} \MKQ g(Z_{\ell})}_{S_1} \\
& \qquad+ \underbrace{\sum_{\ell=1}^{n} \bigl(\Delta^{(0,1)}(\ell,\ell) + w(\ell,\ell)\bigr) M_{\ell} \MKQ g(Z_{\ell-1})}_{S_2} \eqsp,
\end{split}
\end{equation}
and 
\begin{equation}
\label{eq:rem_u_n_def}
\begin{split}
\RemU_{n,rem} 
&= - \MKQ g(Z_n) \sum_{j=1}^{n}w(n,j) M_{j} + \sum_{\ell=2}^{n-1}\bigl( w(\ell,\ell) + w(\ell-1,\ell-1) - w(\ell,\ell-1) \bigr) \{\MKQ g(Z_{\ell-1})\}^2 \\
& + w(1,1) \{\MKQ g(Z_{0})\}^2 + (w(n,n)+w(n-1,n-1)) \{\MKQ g(Z_{n-1})\}^2 \\
& + \sum_{\ell=2}^{n}\Delta^{(m)}(\ell,\ell) \MKQ g(Z_{\ell-1}) \MKQ g(Z_{\ell-2}) + w(n,n-2) \MKQ g(Z_{n-1}) \MKQ g(Z_{n-2}) \\
& - \MKQ g(Z_{n}) \sum_{j=1}^{n} \Delta^{(0,1)}(n,j) \MKQ g(Z_{j-1}) + w(n,n) \{\MKQ g(Z_{n})\}^2 - \sum_{j=n-2}^{n} w(n,j) \MKQ g(Z_{n-1}) \MKQ g(Z_{j})\eqsp.
\end{split}
\end{equation}
\end{lemma}
\begin{proof}
The proof is provided in \Cref{sec:proof_remainder_representation}.
\end{proof}

\section{Applications to the concentration properties of the OBM estimator}
\label{sec:variance}
In this section, we study the concentration properties of the overlapped batch mean (OBM) estimator \cite{meketon1984overlapping, flegal:2010} of the asymptotic variance of the Markov chain. Assume that the Markov kernel $\MKQ$ satisfies the UGE assumption \Cref{assum:UGE}. In this case, for any bounded function $f: \Zset \to \rset$, it is known (see e.g. \cite[Chapter~21]{douc:moulines:priouret:soulier:2018}), that the central limit theorem holds:
\[
\frac{1}{\sqrt{n}}\sum_{k=0}^{n-1}\{f(Z_k) - \pi(f)\} \to \mathcal{N}(0,\sigma_{\infty}^2(f))\eqsp,
\]
where $\sigma_{\infty}^2(f)$ is the \emph{asymptotic variance} defined as 
\begin{equation}
\label{V_infinity_sum} 
\sigma_{\infty}^2(f) = \rho^{(f)}(0) + 2\sum_{\ell=1}^{\infty}\rho^{(f)}(\ell)\eqsp, \text{ where } \rho^{(f)}(\ell) = \PE_{\pi}[\{f(Z_{0}) - \pi(f)\}\{f(Z_{\ell}) - \pi(f)\}] 
\end{equation}
is the $l$-th order autocovariance under stationary. The overlapped batch mean (OBM) estimator \cite{meketon1984overlapping} is defined as
\begin{equation}
\label{eq:OBM_estimator_def}
\hat \sigma_{\OBM}^2(f) = \frac{b_n}{n-b_n+1} \sum_{t = 0}^{n-b_n} ( \pi_{b_n,t}(f) - \pi_{n}(f))^2\eqsp,
\end{equation}
where we have set
\begin{equation}
\pi_{b_n,t}(f) = \frac{1}{b_n}\sum_{\ell=1}^{b_n} f(Z_{t+\ell})\eqsp, \quad t \in \{0,\ldots,n-b_n\}\eqsp, \quad \pi_n(f) = \frac{1}{n} \sum_{\ell=1}^{n} f(Z_\ell)\eqsp.
\end{equation}
It is possible to study the properties of the OBM estimator by reducing it to the spectral estimator with a particular choice of Bartlett weight kernel \cite{welch1987relationship}, \cite{damerdji1991strong}. At the same time, this approach requires to estimate the concentration properties of a number of remainder terms, which we can avoid when proceeding directly with the quadratic form, induced by the OBM estimator. From now on we assume, without loss of generality, that $\pi(f) = 0$. Define the matrix $B = \{B_{ij}\} \in \rset^{n - b_n+1 \times n}$
$$
B_{ij} = \begin{cases}
\frac{1}{b_n},  & i \le j \le i + b_n - 1, \\
0, & \text{ otherwise} \eqsp.
\end{cases}
$$
Let $X = f(Z) = \{f(Z_\ell)\}_{\ell = 1}^{n} \in \rset^{n}$. With these notations, we obtain, using the elementary transformations, that the OBM estimator can be rewritten as 
\begin{align}
\hat \sigma_{\OBM}^2(f) =  \frac{b_n}{n-b_n+1} \left( X^\top B^\top B X - \frac{2}{n} \mathbf{1}_{n-b_n+1}^\top B X \cdot \mathbf{1}_n^\top X + (n-b_n+1) \left( \frac{1}{n} \mathbf{1}_n^\top X \right)^2 \right)\eqsp,
\end{align}
where $\mathbf{1}_k \in \rset^{k \times 1}$ is a column vector of all ones, $k \in \nset$. We define now the column vectors $u, v \in \rset^{n \times 1}$ of the form: 
\begin{equation}
\label{eq:vectors_u_v_def}
u = \frac{1}{n} \1_{n} - \frac{1}{n-b_n+1} B^{\top} \1_{n-b_n+1}\eqsp, \quad 
v = \frac{1}{n-b_n+1} B^{\top} \1_{n-b_n+1}\eqsp.
\end{equation}
With the above notation, we can represent $\hat \sigma_{\OBM}^2(f)$ from \eqref{eq:OBM_estimator_def} as 
\begin{equation}
\label{eq:OBM_equivalent}
\hat \sigma_{\OBM}^2(f) =  \frac{b_n}{n-b_n+1} X^\top B^\top B X + b_n (u^\top X )^2 -  b_n (v^\top X)^2 \eqsp, 
\end{equation}
Next we focus on the quadratic form
$$
V_n(f) = \frac{b_n}{n-b_n+1} X^\top B^\top B X =\sum_{\ell = 1}^{n} \sum_{j = 1}^{\ell} w(\ell, j) f(Z_\ell) f(Z_j),
$$
where $w(\ell, \ell)= \frac{b_n}{n-b_n+1}\{B^\top B\}_{\ell, \ell}$ and $w(\ell, j) = \frac{2b_n}{n-b_n+1}\{B^\top B\}_{\ell, j}$. Note that
$$
w(\ell, k) = \frac{b_n}{n-b_n+1} \sum_{j=1}^{n-b_n+1} B_{j \ell} B_{j k}. 
$$
The diagonal weights in the above formula have the form:
\begin{equation}
\label{eq:w_ell_ell_def}
w(\ell, \ell) =
\begin{cases}
\frac{\ell}{b_n(n - b_n + 1)}, & \text{if } 1 \leq \ell \leq b_n - 1, \\
\frac{1}{n - b_n + 1}, & \text{if } b_n \leq \ell \leq n - b_n + 1, \\
\frac{n - \ell + 1}{b_n(n - b_n + 1)}, & \text{if } n - b_n + 2 \leq \ell \leq n \eqsp.
\end{cases}
\end{equation}
Hence, with direct summation, $\trace(W) = \sum_{\ell=1}^{n}w(\ell,\ell) = 1$.
Similarly, the OBM weights \( w(\ell, j) \), for \( \ell > j \), are given by the following formulas:
\begin{equation}
\label{eq:w_ell_j_def}
w(\ell, j) =
\begin{cases}
\frac{2 j}{b_n (n - b_n + 1)}, 
& \text{if } \ell < b_n,\ 1 \leq j \leq \ell - 1, \\
\frac{2}{b_n (n - b_n + 1)} (b_n - (\ell - j)), 
& \text{if } b_n \leq \ell \le n - b_n + 1,\ \ell - b_n + 1 \leq j \leq \ell - 1, \\
\frac{2}{b_n (n - b_n + 1)} \left( \min(j, n - b_n + 1) - \ell + b_n \right), 
& \text{if } \ell > n - b_n + 1,\ \ell - b_n + 1 \le j \le \ell - 1, \\
0, & \text{otherwise}.
\end{cases}
\end{equation}
The above formulas imply the following upper bounds on the coefficients $w(\ell,j)$ and the quantities introduced in \eqref{eq:coef_def}: 

\begin{lemma}
\label{lem:coef_properties_start}
Let $w(\ell,j)$ be coefficients introduced in \eqref{eq:w_ell_j_def} and \eqref{eq:w_ell_ell_def}. Then it holds that:
\begin{enumerate}
    \item The coefficients $w(\ell,j)$ satisfy 
    \[
    w(\ell,j) \leq 
    \begin{cases}
    \frac{2}{n - b_n+1}, &|\ell - j|\leq b_n, \\
    0, & \text{ otherwise}
    \end{cases}
    \]
    \item Coefficients $\Delta^{(1,0)}(\ell,j)$ and $\Delta^{(0,1)}(\ell,j)$ satisfy, for any $j \leq \ell-1$, the bound
    \begin{equation}
    \label{eq:difference_bound_1_0_0_1}
    \begin{split}
    |\Delta^{(1,0)}(\ell, j)| &\leq \frac{2}{b_n (n-b_n+1)}, \\
    |\Delta^{(0,1)}(\ell, j)| &\leq \frac{2}{b_n (n-b_n+1)}.
    \end{split}
    \end{equation}
    \item Among the coefficients $\Delta^{(1,1)}(\ell,j)$ only 
    \[
    \Delta^{(1,1)}(\ell,\ell-b_n) \neq 0\eqsp, \quad \ell \geq b_n + 1\eqsp.
    \]
    Moreover, it holds that 
    \[
    |\Delta^{(1,1)}(\ell,\ell-b_n)| = \frac{2}{b_n (n-b_n+1)}\eqsp. 
    \]
\end{enumerate}
\end{lemma}
\begin{proof}
Te proof follows from simple algebra and the definition of coefficients $w(\ell,j)$ in \eqref{eq:w_ell_j_def} and coefficients from \eqref{eq:coef_def}. 
\end{proof}

\begin{theorem}
\label{OBM concentration lemma}
    Assume \Cref{assum:UGE} and let $f$ be a bounded function with $\pi(f) = 0$ and $\norm{f}[\infty] \leq 1$. Then for any $p \geq 2$, and $n \geq 2b_n + 1$, it holds that 
    $$
    \PE_{\xi}^{1/p}[|\hat \sigma_{\OBM}^2(f) - \sigma_{\infty}^2(f)|^p] \lesssim \frac{p \taumix^3}{\sqrt{n}} + \frac{p^2 \taumix^2 \sqrt{b_n}}{\sqrt{n}} + \frac{p^2 \taumix^{2}}{\sqrt{b_n}}
    $$
\end{theorem}

\begin{proof}
Recall that, due to the formulas \eqref{eq:key_representation_U_stat} and \eqref{eq:OBM_equivalent}, and the fact that $\sum_{\ell=1}^{n}w(\ell,\ell)=1$, it holds that 
\begin{multline}
\label{eq:D_1_2_def}
\hat \sigma_{\OBM}^2(f) - \sigma_{\infty}^2(f) = \underbrace{\sum_{\ell=1}^{n} w(\ell,\ell) \{ \Delta M_{\ell,\ell} - \sigma_{\infty}^2(f)\}}_{D_1} \\ + \underbrace{\sum_{\ell=1}^{n}\sum_{j=1}^{\ell-1}w(\ell,j) \Delta M_{\ell,j}}_{D_2}  + \RemU_{n} + b_n (u^\top X )^2 -  b_n (v^\top X)^2\eqsp.
\end{multline}
Note that, due to \cite[Chapter~21]{douc:moulines:priouret:soulier:2018}, $\PE_{\pi}\bigl[(g(Z_{\ell}) - \MKQ g(Z_{\ell-1})^2\bigr] = \sigma^2_{\infty}(f)$. In order to proceed further, we introduce the function 
\[
\widehat{g}(z) = \MKQ g^2(z) - (\MKQ g(z))^2\eqsp, \quad z \in \Zset\eqsp.
\]
Note that it holds that $\|g \|_\infty \lesssim \taumix $ and $\|\widehat{g}\|_{\infty} \lesssim \taumix^2$. Then we notice that 
\[
\CPE[\xi]{(g(Z_{\ell}) - \MKQ g(Z_{\ell-1}))^2}{\F_{\ell-1}} = \widehat{g}(Z_{\ell-1})\eqsp,
\]
and we can represent $D_1$ as
\begin{equation}
\label{eq:T_1_split}
D_1 = D_{1,1} + D_{1,2}\eqsp,
\end{equation}
where
\begin{equation}
D_{1,1} = \sum_{\ell=1}^{n-1}w(\ell, \ell)\bigl\{(g(Z_{\ell}) - \MKQ g(Z_{\ell-1}))^2 - \widehat{g}(Z_{\ell-1}) \bigr\}\eqsp, \quad 
D_{1,2} = \sum_{\ell=1}^{n} w(\ell, \ell)\{\widehat{g}(Z_{\ell-1}) - \sigma^2_{\infty}(f)\}\eqsp.
\end{equation}
Now we estimate the moments of $D_{1,1}$ and $D_{1,2}$ separately. It is easy to see that $D_{1,1}$ is a weighted sum of martingale-difference sequence w.r.t. $\F_{\ell-1}$, thus, using Burkholder's inequality \cite[Theorem~8.6]{osekowski:2012}, we get that 
\begin{align}
\PE_{\xi}^{1/p}[|D_{1,1}|^p] &\leq p \PE_{\xi}^{1/p}[\bigl(\sum_{\ell=1}^{n} w^2(\ell, \ell)\bigl\{(g(Z_{\ell}) - \MKQ g(Z_{\ell-1}))^2 - \widehat{g}(Z_{\ell-1})\bigr\}^2 \bigr)^{p/2}]\\
&\lesssim   p \taumix^2 (\sum_{\ell=1}^{n} w^2(\ell, \ell) \bigr)^{1/2} \leq \frac{p \taumix^2}{(n-b_n+1)^{1/2}}\eqsp.
\end{align}
In the last inequality we have used the fact that 
\[
\sum_{\ell = 1}^{n} w^2(\ell, \ell) = 
\frac{1}{(n - b_n + 1)^2} \left[
\frac{(b_n - 1)(2b_n - 1)}{3 b_n} + (n - 2b_n + 2)
\right] \leq \frac{1}{n-b_n+1}\eqsp.
\]
Similarly, applying \Cref{lem:auxiliary_rosenthal_weighted}, we get using expression for the weights $w(\ell,\ell)$, that  
\[
\PE_{\xi}^{1/p}[|D_{1,2}|^{p}] \lesssim \frac{p^{1/2} \taumix^{3}}{(n-b_n+1)^{1/2}} + \frac{ \taumix^{3}}{n-b_n+1} \lesssim \frac{p^{1/2} \taumix^{3}}{(n-b_n+1)^{1/2}}\eqsp.
\]
Combining the bounds for $D_{1,1}$ and $D_{1,2}$, we end up with the following bound:
\begin{equation}
\label{eq:D_1_bound}
\PE_{\xi}^{1/p}[|D_{1}|^{p}] \lesssim  \frac{p \taumix^2 + p^{1/2} \taumix^3}{(n-b_n+1)^{1/2}}\eqsp.
\end{equation}
Now we proceed further and our next step is to provide an upper bound on the term $D_2$ from \eqref{eq:D_1_2_def} defined as 
\begin{equation}
\label{eq:D_2_def_u_stat}
D_2 = \sum_{\ell=1}^{n}\sum_{j=1}^{\ell-1}w(\ell,j) \Delta M_{\ell,j} = \sum_{\ell=1}^{n} (g(Z_{\ell}) - \MKQ g(Z_{\ell-1})) \sum_{j=1}^{\ell-1}w(\ell,j)  (g(Z_j) - \MKQ g(Z_{j-1}))\eqsp.
\end{equation}
Since the sum above is a weighted sum of a martingale-difference sequence, we get using Burkholder's inequality \cite[Theorem~8.6]{osekowski:2012} that 
\begin{align}
\label{eq:T_2_bound_martingale_diff}
\PE_{\xi}^{1/p}[|D_2|^{p}] 
&\leq p \PE_{\xi}^{1/p}\bigl[ \bigl(\sum_{\ell=1}^{n} M_{\ell}^{2}\bigl\{\sum_{j=1}^{\ell-1}w(\ell,j)  M_{j}\bigr\}^{2}\bigr)^{p/2}\bigr] \\
&\leq p \biggl\{\sum_{\ell=1}^{n}\PE^{2/p}_{\xi}\bigl[ M_{\ell}^{p} \bigl|\sum_{j=1}^{\ell-1}w(\ell,j)  M_{j}\bigr|^{p} \bigr]\biggr\}^{1/2}\eqsp.
\end{align}
Proceeding now with the inner terms in the sum above, we get 
\begin{align}
\PE^{2/p}_{\xi}\bigl[ M_{\ell}^{p} \bigl| \sum_{j=1}^{\ell-1}w(\ell,j)  M_{j}\bigr|^{p} \bigr] \leq \PE^{1/p}_{\xi}\bigl[|M_{\ell}|^{2p}\bigr] \PE^{1/p}_{\xi}\bigl[\bigl|\sum_{j=1}^{\ell-1}w(\ell,j)  M_{j}\bigr|^{2p}\bigr]\eqsp.
\end{align}
Now we can apply again Burkholder's inequality to the second term above and obtain 
\begin{align}
\PE^{1/2p}_{\xi}\bigl[\bigl| \sum_{j=1}^{\ell-1}w(\ell,j)  M_{j}\bigr|^{2p}\bigr] 
&\leq 2p \PE_{\xi}^{1/2p}\bigl[\bigl(\sum_{j=1}^{\ell-1}w^2(\ell,j) M_{j}^2\bigr)^{p}\bigr] \\
&\lesssim p \taumix \bigl(\sum_{j=1}^{\ell-1}w^2(\ell,j)\bigr)^{1/2} \\
&\lesssim \frac{p \taumix b_{n}^{1/2}}{n-b_n+1}\eqsp.
\end{align}
Combining the above bounds, we obtain the following inequality:
\begin{align}
\PE^{2/p}_{\xi}\bigl[ M_{\ell}^{p} \bigl| \sum_{j=1}^{\ell-1}w(\ell,j)  M_{j}\bigr|^{p} \bigr] \leq  \frac{p^2 \taumix^4 b_n}{(n-b_n+1)^2}\eqsp,
\end{align}
and, combining everything inside \eqref{eq:T_2_bound_martingale_diff}, we obtain 
\begin{align}
\label{eq:D_2_bound}
\PE_{\xi}^{1/p}[|D_2|^{p}] \lesssim \frac{p^2 \taumix^2 \sqrt{b_n}}{\sqrt{n-b_n+1}}\eqsp.
\end{align}
Now it remains to bound the moments of the term $\RemU_{n}$. Using \Cref{lem:bar_r_n_bound},
\begin{equation}
\label{eq:Rem_U_bound}
\PE_{\xi}^{1/p}[|\RemU_{n}|^{p}] \lesssim \frac{p \taumix^{2}}{\sqrt{b_n}} + \frac{p \taumix^2 \sqrt{b_n}} {n - b_n+1} + \frac{\taumix^2}{b_n} \eqsp.
\end{equation}
Our proof of the bound \eqref{eq:Rem_U_bound} is based on re-arranging the terms in the representation \eqref{eq:RemU_definition} in a way suggested in \Cref{lem:bar_r_n_representation}. It remains to bound 
$\PE_{\xi}^{1/p}[|u^\top X|^{2p}]$ and $\PE_{\xi}^{1/p}[|v^\top X|^{2p}]$, where the vectors $u$ and $v$ are defined in \eqref{eq:vectors_u_v_def}. Using the expressions in \eqref{eq:vectors_u_v_def}, we obtain that the elements of $u$ and $v$ are written as 
\[
v_\ell =
\begin{cases}
\frac{\ell}{b_n(n - b_n + 1)}, & \text{if } 1 \leq \ell \leq b_n - 1\eqsp, \\
\frac{1}{n - b_n + 1}, & \text{if } b_n \leq \ell \leq n - b_n + 1, \\
\frac{n - \ell + 1}{b_n(n - b_n + 1)}, & \text{if } n - b_n + 2 \leq \ell \leq n\eqsp,
\end{cases}
\]
and $u_{\ell} = 1/n - v_{\ell}$. Hence, with direct summation we obtain that 
\[
|v_1| + |v_n| + \sum_{k=1}^{n-1}|v_{k+1}-v_k| \lesssim \frac{1}{n}\eqsp, \quad |u_1| + |u_n| + \sum_{k=1}^{n-1}|u_{k+1}-u_k| \lesssim \frac{1}{n}\eqsp,
\]
and 
\[
\sum_{\ell=1}^{n}v_{\ell}^2 \lesssim \frac{1}{n}\eqsp, \quad \sum_{\ell=1}^{n}u_{\ell}^2 = \sum_{\ell=1}^{n}v_{\ell}^2 - 1/n  \lesssim \frac{1}{n}\eqsp.
\]
Then 
\Cref{lem:auxiliary_rosenthal_weighted}  implies that 
\begin{equation}
\label{eq:rosenthal_inequality_applied}
\PE_{\xi}^{1/p}[|v^\top X|^{2p}] \lesssim \frac{p\taumix}{n} + \frac{\taumix^2}{n^2}\eqsp, 
\end{equation}
and the same bound (up to constant factors) on $\PE_{\xi}^{1/p}[|u^\top X|^{2p}]$. To conclude the proof it remains to combine the bounds \eqref{eq:rosenthal_inequality_applied}, \eqref{eq:Rem_U_bound}, \eqref{eq:D_2_bound}, and \eqref{eq:D_1_bound} with the representation \eqref{eq:D_1_2_def}.
\end{proof}

\section{Postponed proofs}
\label{appendix:quadratic}
\subsection{Proof of \Cref{prop:quadr_form_decomposition}}
\label{sec:prrof_quadr_form_decomposition}
Introduce the following notations:
\begin{equation}
\label{eq:Markov_kernel_doubled_defi}
\bar{\pi}(\rmd u, \rmd v) = \pi(\rmd u) \pi(\rmd v)\eqsp, \quad 
\MKR_2(x,y,du,dv) = \sum_{m = 0}^{\infty} \sum_{k = 0}^{\infty}\bigl(\MKQ^{m}(x,\rmd u) - \pi(\rmd u)\bigr)\bigl(\MKQ^{k}(y,\rmd v) - \pi(\rmd v)\bigr)\eqsp.
\end{equation}
Note that in the above expression for any $x,y \in \Zset$, $\MKR_2(x,y,du,dv)$ is a finite (signed) measure. Let $h(x, y) = f(x) f(y)$. Then $\bar{\pi}(h) = 0$. Following the notations introduced in \cite{atchade2014martingale}, we obtain that 
\begin{align}
\label{eq:bar_G_2_def}
&\bar{h}_1(x) = \int h(x,z)\pi(\rmd z) = 0\eqsp, \quad \bar{h}_2(x,y) = h(x,y) - \bar{h}_1(x) - \bar{h}_1(y) = h(x,y)\eqsp, \\
&\bar{G}_2(x,y) = \int \int \bar{h}_2(u,v) \MKR_2(x,y,\rmd u,\rmd v) = g(x)g(y)\eqsp, \\
&\MKQ \bar{G}_2(x,y) = g(y) \MKQ g(x)\eqsp, \quad \MKQ^{2} \bar{G}_2(x,y) = \MKQ g(x) \MKQ g(y)\eqsp.
\end{align}  
Introduce also the function
\begin{equation}
\label{eq:Lambda_2_def}
\begin{split}
\Lambda_2(x_1,x_2,y_1,y_2)
&= \bar{G}_2(y_1,y_2) - \MKQ \bar{G}_2(x_2,y_1) - \MKQ \bar{G}_2(x_1,y_2) + \MKQ^2 \bar{G}_2(x_1,x_2) \\
&= g(y_1)g(y_2) - g(y_1) \MKQ g(x_2) - g(y_2) \MKQ g(x_1) + \MKQ g(x_1)  \MKQ g(x_2) \\
&= (g(y_1) - \MKQ g(x_1))(g(y_2) - \MKQ g(x_2))\eqsp.
\end{split}
\end{equation}
With this notation for $\Lambda_2(x_1,x_2,y_1,y_2)$ we have the following key relation:
\[
h(x,y) = f(x)f(y) = (g(x)-\MKQ g(x))(g(y)-\MKQ g(y)) = \Lambda_2(x,y,x,y)\eqsp.
\]
The rest of the proof follows from \cite[Lemma~2.2]{atchade2014martingale}.

\subsection{Proof of the error representation from \Cref{lem:bar_r_n_representation}}
\label{sec:proof_remainder_representation}
Recall that the remainder term $\RemU_{n}$ is decomposed into a sum of three terms, denoted, respectively, as $T_1$, $T_2$, and $T_3$, and outlined in \eqref{eq:reminder_u_stat}. Now we proceed with $T_1$, $T_2$, $T_3$ separately. We first consider the term $T_1$ and write:
\begin{align}
T_1 = \underbrace{\sum_{\ell=1}^{n}\sum_{j=1}^{\ell} w(\ell,j) g(Z_{j}) \MKQ g(Z_{\ell-1})}_{T_{1,1}} - \underbrace{\sum_{\ell=1}^{n}\sum_{j=1}^{\ell} w(\ell,j) g(Z_{j}) \MKQ g(Z_{\ell})}_{T_{1,2}}\eqsp,
\end{align}
and now transform $T_{1,1}$, using the definition of $M_{j}$ in \eqref{eq:martingale_increment_def}:
\begin{align}
T_{1,1}
&= \sum_{\ell=1}^{n} \MKQ g(Z_{\ell-1}) \sum_{j=1}^{\ell} w(\ell,j) M_{j} +  \sum_{\ell=1}^{n} \MKQ g(Z_{\ell-1}) \sum_{j=1}^{\ell} w(\ell,j) \MKQ g(Z_{j-1}) \\
&= \sum_{\ell=2}^{n} \MKQ g(Z_{\ell-1}) \sum_{j=1}^{\ell-1} w(\ell,j) M_{j}  + \sum_{\ell=1}^{n} w(\ell,\ell) \MKQ g(Z_{\ell-1}) M_{\ell} + \sum_{\ell=1}^{n} \MKQ g(Z_{\ell-1}) \sum_{j=1}^{\ell} w(\ell,j) \MKQ g(Z_{j-1})\eqsp.
\end{align}
Similarly, we transform $T_{1,2}$ into 
\begin{align}
T_{1,2} &= \underbrace{\sum_{\ell=1}^{n} \MKQ g(Z_{\ell}) \sum_{j=1}^{\ell} w(\ell,j)  M_{j}}_{T_{1,2,1}} + \underbrace{\sum_{\ell=1}^{n} \MKQ g(Z_{\ell}) \sum_{j=1}^{\ell} w(\ell,j) \MKQ g(Z_{j-1})}_{T_{1,2,2}}\eqsp,
\end{align}
and we further transform $T_{1,2,1}$ as follows:
\begin{align}
T_{1,2,1} 
&=
\sum_{\ell=2}^{n+1}  \MKQ g(Z_{\ell-1})  \sum_{j=1}^{\ell-1} w(\ell-1,j)  M_{j} \\
&= \sum_{\ell=2}^{n}  \MKQ g(Z_{\ell-1})  \sum_{j=1}^{\ell-1} w(\ell-1,j)  M_{j} + \MKQ g(Z_n) \sum_{j=1}^{n}w(n,j) M_{j}\eqsp.
\end{align}
Hence, combining the above expressions together, we obtain that 
\begin{align}
\label{eq:T_1_decomposition_filled}
T_1 = \sum_{\ell=2}^{n} \MKQ g(Z_{\ell-1}) \sum_{j=1}^{\ell-1} \Delta^{(1,0)}(\ell,j) M_{j} + V_{1} + \RemU_{n,1} \eqsp,
\end{align}
where we have set
\begin{equation}
\label{eq:Delta_1_0_def}
\Delta^{(1,0)}(\ell,j) = w(\ell,j) - w(\ell-1,j)\eqsp,
\end{equation}
and
\begin{equation}
\label{eq:RemU_definition}
\begin{split}
V_{1} &= \sum_{\ell=1}^{n} \MKQ g(Z_{\ell-1}) \sum_{j=1}^{\ell} w(\ell,j) \MKQ g(Z_{j-1}) - \sum_{\ell=1}^{n} \MKQ g(Z_{\ell}) \sum_{j=1}^{\ell} w(\ell,j) \MKQ g(Z_{j-1}) \\
\RemU_{n,1} &= \sum_{\ell=1}^{n} w(\ell,\ell) \MKQ g(Z_{\ell-1}) M_{\ell} - \MKQ g(Z_n) \sum_{j=1}^{n}w(n,j) M_{j}\eqsp.
\end{split}
\end{equation}
We prefer here to extract one more diagonal term from $T_{1}$:
\begin{align}
\label{eq:T_1_decomposition_final}
T_1 &= \sum_{\ell=3}^{n} \MKQ g(Z_{\ell-1}) \sum_{j=1}^{\ell-2} \Delta^{(1,0)}(\ell,j) M_{j} \\
&\qquad \qquad \qquad \qquad +\sum_{\ell=2}^{n} \Delta^{(1,0)}(\ell,\ell-1) \MKQ g(Z_{\ell-1}) M_{\ell-1} + V_{1} + \RemU_{n,1} \eqsp,
\end{align}
where the terms $V_{1}$ and $\RemU_{n,1}$ are defined in \eqref{eq:RemU_definition}. Now we proceed with the term $T_2$. Similarly to the above decomposition of $T_1$, we can write:
\begin{align}
T_2 = \underbrace{\sum_{\ell=1}^{n} M_{\ell} \sum_{j=1}^{\ell} w(\ell,j) \left(\MKQ g(Z_{j-1}) - \MKQ g(Z_{j})\right)}_{T_{2,1}} + \underbrace{\sum_{\ell=1}^{n} \sum_{j=1}^{\ell} w(\ell,j) \MKQ g(Z_{\ell-1}) \left(\MKQ g(Z_{j-1}) - \MKQ g(Z_{j})\right)}_{T_{2,2}}\eqsp. 
\end{align}
Now, proceeding with the term $T_{2,1}$, we obtain that
\begin{align}
T_{2,1} 
&= \sum_{\ell=1}^{n} M_{\ell} \sum_{j=1}^{\ell} w(\ell,j) \MKQ g(Z_{j-1}) 
 - \sum_{\ell=1}^{n} M_{\ell} \sum_{j=1}^{\ell} w(\ell,j)  \MKQ g(Z_{j}) \\
 &= \sum_{\ell=1}^{n} M_{\ell} \sum_{j=1}^{\ell} w(\ell,j) \MKQ g(Z_{j-1}) - \sum_{\ell=1}^{n} M_{\ell} \sum_{j=1}^{\ell} w(\ell,j-1)  \MKQ g(Z_{j-1}) \\
 & \qquad -  \sum_{\ell=1}^{n} w(\ell,\ell) M_{\ell} \MKQ g(Z_{\ell}) + 
 \MKQ g(Z_0) \sum_{\ell=1}^{n} w(\ell,0)  M_{\ell} \\
 &= \sum_{\ell=1}^{n} M_{\ell} \sum_{j=1}^{\ell} \Delta^{(0,1)}(\ell,j) \MKQ g(Z_{j-1}) - \sum_{\ell=1}^{n} w(\ell,\ell) M_{\ell} \MKQ g(Z_{\ell}) + 
 \MKQ g(Z_0) \sum_{\ell=1}^{n} w(\ell,0)  M_{\ell}\eqsp,
\end{align}
where we have set 
\begin{align}
\label{eq:Delta_0_1_def}
\Delta^{(0,1)}(\ell,j) = w(\ell,j) - w(\ell,j-1)\eqsp.
\end{align}
Hence, combining the above decomposition, we get
\begin{align}
\label{eq:T_2_def_r_n_decomposition}
T_2 = \sum_{\ell=1}^{n} M_{\ell} \sum_{j=1}^{\ell-1} \Delta^{(0,1)}(\ell,j) \MKQ g(Z_{j-1}) + V_{2} + \RemU_{n,2} \eqsp,
\end{align}
where we have set 
\begin{equation}
\label{eq:RemU_2_definition}
\begin{split}
V_{2} &= \sum_{\ell=1}^{n} \sum_{j=1}^{\ell} w(\ell,j) \MKQ g(Z_{\ell-1}) \left(\MKQ g(Z_{j-1}) - \MKQ g(Z_{j})\right)\eqsp, \\
\RemU_{n,2} &= -\sum_{\ell=1}^{n} w(\ell,\ell) M_{\ell} \MKQ g(Z_{\ell}) + \sum_{\ell=1}^{n} \Delta^{(0,1)}(\ell,\ell) M_{\ell} \MKQ g(Z_{\ell-1})\eqsp,
\end{split}
\end{equation}
and used the convention $w(\ell,0) = 0$. Combining now the terms $V_1,V_2$ from \eqref{eq:T_1_decomposition_final} and \eqref{eq:T_2_def_r_n_decomposition}, together with $T_3$ from \eqref{eq:reminder_u_stat}, we obtain:
\begin{align}
\label{eq:V_1_V_2_T_3_sum}
V_1 + V_2 + T_3 &= \underbrace{\sum_{\ell=1}^{n} \MKQ g(Z_{\ell-1}) \sum_{j=1}^{\ell} w(\ell,j) \MKQ g(Z_{j-1})}_{Q_1} - \underbrace{\sum_{\ell=1}^{n} \MKQ g(Z_{\ell}) \sum_{j=1}^{\ell} w(\ell,j) \MKQ g(Z_{j-1})}_{Q_2} \\
&- \underbrace{\sum_{\ell=1}^{n} \sum_{j=1}^{\ell} w(\ell,j) \MKQ g(Z_{\ell-1}) \MKQ g(Z_{j})}_{Q_3} + \underbrace{\sum_{\ell=1}^{n}\sum_{j=1}^{\ell} w(\ell,j) \MKQ g(Z_{j}) \MKQ g(Z_{\ell})}_{Q_4}\eqsp.
\end{align}
Note that the expression above is a quadratic form in random variables $\{\MKQ g(Z_{\ell})\}_{\ell=0}^{n}$, and we can re-write it in a more compact way, setting $t_{\ell} := \MKQ g(Z_{\ell})$. In all expressions below we follow the convention $\sum_{\ell=m}^{n}a_{\ell} = 0$, provided that $n < m$. 
\begin{align}
Q_1 &= \sum_{\ell=2}^{n}\sum_{j=1}^{\ell-1} t_{\ell-1} w(\ell,j) t_{j-1} + \sum_{\ell=1}^{n}w(\ell,\ell) t_{\ell-1}^2 \\
&= \sum_{\ell=3}^{n}\sum_{j=1}^{\ell-2} t_{\ell-1} w(\ell,j) t_{j-1} + \sum_{\ell=1}^{n}w(\ell,\ell) t_{\ell-1}^2 + \sum_{\ell=2}^{n}  w(\ell,\ell-1) t_{\ell-1} t_{\ell-2}\eqsp.
\end{align}
Similarly, with elementary algebra:
\begin{align}
Q_2 &= \sum_{\ell=1}^{n} \sum_{j=1}^{\ell} w(\ell,j) t_{\ell} t_{j-1} 
= \sum_{\ell=1}^{n-1} \sum_{j=1}^{\ell} w(\ell,j) t_{\ell} t_{j-1} + \sum_{j=1}^{n} w(n,j) t_{n} t_{j-1} \\
&= \sum_{\ell=2}^{n} \sum_{j=1}^{\ell-1} w(\ell-1,j) t_{\ell-1} t_{j-1} + \sum_{j=1}^{n} w(n,j) t_{n} t_{j-1} \\
&= \sum_{\ell=3}^{n} \sum_{j=1}^{\ell-2} w(\ell-1,j) t_{\ell-1} t_{j-1} + \sum_{\ell=1}^{n-1} w(\ell,\ell) t_{\ell} t_{\ell-1} + \sum_{j=1}^{n} w(n,j) t_{n} t_{j-1}\eqsp.
\end{align}

\begin{align}
Q_3 &= \sum_{\ell=1}^{n} \sum_{j=1}^{\ell} w(\ell,j) t_{\ell-1} t_{j} = \sum_{\ell=1}^{n-1} \sum_{j=1}^{\ell} w(\ell,j) t_{\ell-1} t_{j} + \sum_{j=1}^{n} w(n,j) t_{n-1} t_{j} \\
&= \sum_{\ell=1}^{n-1} \sum_{j=2}^{\ell+1} w(\ell,j-1) t_{\ell-1} t_{j-1} + \sum_{j=1}^{n} w(n,j) t_{n-1} t_{j} \\
&= \sum_{\ell=3}^{n-1} \sum_{j=2}^{\ell-1} w(\ell,j-1) t_{\ell-1} t_{j-1} + \sum_{\ell=2}^{n-1} w(\ell,\ell-1) t_{\ell-1}^2 + \sum_{\ell=1}^{n-1} w(\ell,\ell) t_{\ell-1} t_{\ell} + \sum_{j=1}^{n} w(n,j) t_{n-1} t_{j} \\
&= \sum_{\ell=3}^{n-1} \sum_{j=1}^{\ell-2} w(\ell,j-1) t_{\ell-1} t_{j-1} + \sum_{\ell=3}^{n-1}w(\ell,\ell-2) t_{\ell-1}t_{\ell-2} - \sum_{\ell=3}^{n-1}w(\ell,0) t_{\ell-1} t_{0} \\
& \qquad + \sum_{\ell=2}^{n-1} w(\ell,\ell-1) t_{\ell-1}^2 + \sum_{\ell=1}^{n-1} w(\ell,\ell) t_{\ell-1} t_{\ell} + \sum_{j=1}^{n} w(n,j) t_{n-1} t_{j} \\
&= \sum_{\ell=3}^{n} \sum_{j=1}^{\ell-2} w(\ell,j-1) t_{\ell-1} t_{j-1} + \sum_{\ell=3}^{n-1}w(\ell,\ell-2) t_{\ell-1}t_{\ell-2} - \sum_{\ell=3}^{n-1}w(\ell,0) t_{\ell-1} t_{0} - \sum_{j=1}^{n-2} w(n,j-1) t_{n-1} t_{j-1} \\
&\qquad + \sum_{\ell=2}^{n-1} w(\ell,\ell-1) t_{\ell-1}^2 + \sum_{\ell=1}^{n-1} w(\ell,\ell) t_{\ell-1} t_{\ell} + \sum_{j=1}^{n} w(n,j) t_{n-1} t_{j}\eqsp.
\end{align}
Finally, we proceed with the expression for the term $Q_4$:
\begin{align}
Q_4 &= \sum_{\ell=1}^{n}\sum_{j=1}^{\ell} w(\ell,j) t_{\ell} t_{j}
= \sum_{\ell=2}^{n+1}\sum_{j=2}^{\ell} w(\ell-1,j-1) t_{\ell-1} t_{j-1} \\
&= \sum_{\ell=2}^{n}\sum_{j=2}^{\ell} w(\ell-1,j-1) t_{\ell-1} t_{j-1} + \sum_{j=2}^{n+1} w(n,j-1) t_{n} t_{j-1} \\
&= \sum_{\ell=3}^{n}\sum_{j=2}^{\ell-1} w(\ell-1,j-1) t_{\ell-1} t_{j-1} + \sum_{\ell=2}^{n}w(\ell-1,\ell-1) t_{\ell-1}^2 + \sum_{j=2}^{n+1} w(n,j-1) t_{n} t_{j-1} \\
&= \sum_{\ell=3}^{n}\sum_{j=2}^{\ell-2} w(\ell-1,j-1) t_{\ell-1} t_{j-1} + \sum_{\ell=3}^{n}w(\ell-1,\ell-2) t_{\ell-1} t_{\ell-2}  \\
&\qquad \qquad +\sum_{\ell=2}^{n}w(\ell-1,\ell-1) t_{\ell-1}^2 + \sum_{j=2}^{n+1} w(n,j-1) t_{n} t_{j-1} \\
&= 
\sum_{\ell=3}^{n}\sum_{j=1}^{\ell-2} w(\ell-1,j-1) t_{\ell-1} t_{j-1} - \sum_{\ell=3}^{n} w(\ell-1,0) t_{\ell-1} t_{0} + \sum_{\ell=3}^{n}w(\ell-1,\ell-2) t_{\ell-1} t_{\ell-2} \\
&\qquad \qquad +\sum_{\ell=2}^{n}w(\ell-1,\ell-1) t_{\ell-1}^2 + \sum_{j=2}^{n+1} w(n,j-1) t_{n} t_{j-1}\eqsp.
\end{align}
Combining the above expressions for $Q_1,Q_2,Q_3,Q_4$, we obtain that 
\begin{align}
Q_1 - Q_2 - Q_3 + Q_4 = \sum_{\ell=3}^{n} \sum_{j=1}^{\ell-2} \Delta^{(1,1)}(\ell,j) t_{\ell-1} t_{j-1} + \RemU_{n,3}\eqsp,
\end{align}
where we have set the coefficients
\[
\Delta^{(1,1)}(\ell,j) = w(\ell,j) - w(\ell,j-1) - w(\ell-1,j) + w(\ell-1,j-1)\eqsp,
\]
and the remainder term 
\begin{align}
\RemU_{n,3} 
&= \sum_{\ell=1}^{n}w(\ell,\ell) t_{\ell-1}^2 + \sum_{\ell=2}^{n}  w(\ell,\ell-1) t_{\ell-1} t_{\ell-2}  - \sum_{\ell=1}^{n-1} w(\ell,\ell) t_{\ell} t_{\ell-1} - \sum_{j=1}^{n} w(n,j) t_{n} t_{j-1} \\
& - \sum_{\ell=3}^{n-1}w(\ell,\ell-2) t_{\ell-1}t_{\ell-2} + \sum_{\ell=3}^{n-1}w(\ell,0) t_{\ell-1} t_{0} + \sum_{j=1}^{n-2} w(n,j-1) t_{n-1} t_{j-1} \\
& - \sum_{\ell=2}^{n-1} w(\ell,\ell-1) t_{\ell-1}^2 - \sum_{\ell=1}^{n-1} w(\ell,\ell) t_{\ell-1} t_{\ell} - \sum_{j=1}^{n} w(n,j) t_{n-1} t_{j} - \sum_{\ell=3}^{n} w(\ell-1,0) t_{\ell-1} t_{0} \\
& + \sum_{\ell=3}^{n}w(\ell-1,\ell-2) t_{\ell-1} t_{\ell-2} + \sum_{\ell=2}^{n}w(\ell-1,\ell-1) t_{\ell-1}^2 + \sum_{j=2}^{n+1} w(n,j-1) t_{n} t_{j-1}\eqsp.
\end{align}
With the algebraic manipulations on $\RemU_{n,3}$ defined above, we obtain, using in addition that $w(\ell,0) = w(0,\ell) = 0$ for any $\ell$, that 
\begin{equation}
\begin{split}
\label{eq:rem_U_3_def}
\RemU_{n,3} 
&= \sum_{\ell=2}^{n-1}\bigl( w(\ell,\ell) + w(\ell-1,\ell-1) - w(\ell,\ell-1) \bigr)  t_{\ell-1}^2 + w(1,1)t_0^2 + (w(n,n)+w(n-1,n-1))t_{n-1}^2 \\
& + \sum_{\ell=2}^{n}\Delta^{(m)}(\ell,\ell) t_{\ell-1}t_{\ell-2} + w(n,n-2)t_{n-1}t_{n-2} \\
& - \sum_{j=1}^{n} \Delta^{(0,1)}(n,j) t_{n} t_{j-1} + w(n,n) t_n^2 - \sum_{j=n-2}^{n} w(n,j) t_{n-1} t_{j}\eqsp.
\end{split}
\end{equation}
In order to complete the proof, it remains to combine \eqref{eq:rem_U_3_def}, \eqref{eq:T_2_def_r_n_decomposition}, and \eqref{eq:T_1_decomposition_final}.

\begin{lemma}
\label{lem:bar_r_n_remainder_bound}
Under assumptions of \Cref{OBM concentration lemma}, it holds for any $p \geq 2$ that 
\begin{equation}
\label{eq:bar_r_n_reminader_bound}
\PE_{\xi}^{1/p}|\RemU_{n,rem}|^{p} \lesssim \frac{p \taumix^2 \sqrt{b_n}} {n - b_n+1} + \frac{\taumix^2}{b_n}\eqsp,
\end{equation}
where the remainder term $\RemU_{n,rem}$ is defined in \eqref{eq:rem_u_n_def}.
\end{lemma}
\begin{proof}
Our proof is based on the error representation \eqref{eq:bar_r_n_representation} for the remainder term $\RemU_{n,rem}$. We first notice that 
\begin{equation}
\label{eq:aux_bounds_remainder}
|M_{j}| \lesssim \taumix\eqsp, \quad |\MKQ g(Z_{j})| \lesssim \taumix \eqsp.
\end{equation}
Now we consider the error decomposition \eqref{eq:rem_u_n_def}: 
\begin{equation}
\label{eq:rem_u_n_def_appendix}
\begin{split}
\RemU_{n,rem} 
&= \RemU_{n,rem,1} + \RemU_{n,rem,2} + \RemU_{n,rem,3} + \RemU_{n,rem,4} + \RemU_{n,rem,5}\eqsp,
\end{split}
\end{equation}
where we have defined 
\begin{equation}
\label{eq:rem_u_n_reminders}
\begin{split}
\RemU_{n,rem,1} &= - \MKQ g(Z_n) \sum_{j=1}^{n}w(n,j) M_{j}\eqsp, \\
\RemU_{n,rem,2} &= \sum_{\ell=2}^{n-1}\bigl( w(\ell,\ell) + w(\ell-1,\ell-1) - w(\ell,\ell-1) \bigr) \{\MKQ g(Z_{\ell-1})\}^2\eqsp, \\
\RemU_{n,rem,3} &= w(1,1) \{\MKQ g(Z_{0})\}^2 + (w(n,n)+w(n-1,n-1)) \{\MKQ g(Z_{n-1})\}^2 \\
&+ w(n,n-2) \MKQ g(Z_{n-1}) \MKQ g(Z_{n-2}) + w(n,n) \{\MKQ g(Z_{n})\}^2 - \sum_{j=n-2}^{n} w(n,j) \MKQ g(Z_{n-1}) \MKQ g(Z_{j})\eqsp, \\
\RemU_{n,rem,4} &= \sum_{\ell=2}^{n}\Delta^{(m)}(\ell,\ell) \MKQ g(Z_{\ell-1}) \MKQ g(Z_{\ell-2})\eqsp, \quad \RemU_{n,rem,5} = -\MKQ g(Z_{n}) \sum_{j=1}^{n} \Delta^{(0,1)}(n,j) \MKQ g(Z_{j-1})\eqsp.
\end{split}
\end{equation}
Using Burkholder's inequality and \eqref{eq:aux_bounds_remainder}, that for any $p \geq 2$, 
\begin{align}
\PE_{\xi}^{1/p}\bigl[\bigl| \RemU_{n,rem,1} \bigr|^{1/p}\bigr] 
&\lesssim \taumix \PE_{\xi}^{1/p}[\bigl| \sum_{j=1}^{n}w(n,j) M_{j} \bigr|^{p}] \\
&\lesssim \taumix p \PE_{\xi}^{1/p}[\bigl( \sum_{j=1}^{n}w^2(n,j) M_{j}^2 \bigr)^{p/2}] \\
&\lesssim \frac{p \taumix^2 \sqrt{b_n}} {n - b_n+1}\eqsp.
\end{align}
Applying Minkowski's inequality and the bound \eqref{eq:differences_check_w_ell}, we get 
\begin{align}
\PE_{\xi}^{1/p}\bigl[\bigl| \RemU_{n,rem,2} \bigr|^{1/p}\bigr] \lesssim \frac{\taumix^2}{b_n}\eqsp.
\end{align}
Similarly, using Minkowski's inequality, we show that
\begin{equation}
\PE_{\xi}^{1/p}\bigl[|\RemU_{n,rem,3}|^{p} \bigr] \lesssim \frac{\taumix^2}{b_n (n-b_n+1)}\eqsp.   
\end{equation}
Using Minkowski's inequality and the bound $|\Delta^{(m)}(\ell,\ell)| \leq \frac{2}{b_n (n-b_n+1)}$, we get that 
\begin{equation}
\PE_{\xi}^{1/p}\bigl[|\RemU_{n,rem,4}|^{p} \bigr] \lesssim \frac{\taumix^2}{b_n}\eqsp.   
\end{equation}
Using Minkowski's inequality, bound \eqref{eq:difference_bound_1_0_0_1}, and the fact that there are at most $b_n$ non-zero terms among $\Delta^{(0,1)}(n,j)$, $1 \leq j \leq n$, we get 
\begin{equation}
\PE_{\xi}^{1/p}\bigl[|\RemU_{n,rem,5}|^{p} \bigr] \lesssim \frac{\taumix^2}{n-b_n+1}\eqsp.
\end{equation}
It remains to combine the above bounds.
\end{proof}

\begin{lemma}
\label{lem:bar_r_n_martigale_bound}
Under assumptions of \Cref{OBM concentration lemma}, it holds for any $p \geq 2$ that 
\begin{equation}
\label{eq:bar_r_n_martingale_bound}
\PE_{\xi}^{1/p}|\RemU_{n,mart}|^{p} \lesssim \frac{\taumix^2}{b_n}\eqsp.
\end{equation}
\end{lemma}
\begin{proof}
We recall that 
\[
\RemU_{n,mart} = S_1 + S_2\eqsp,
\]
where $S_1$ and $S_2$ are defined in \eqref{eq:rem_u_mart_def}, and analyze these terms separately. We first note that the term $S_1$ can be written as 
\begin{align}
S_1 &= \sum_{\ell=2}^{n} \Delta^{(1,0)}(\ell,\ell-1) \MKQ g(Z_{\ell-1}) M_{\ell-1} - \sum_{\ell=1}^{n} w(\ell,\ell) M_{\ell} \MKQ g(Z_{\ell}) \\
&= \sum_{\ell=2}^{n}\{w(\ell,\ell-1) - w(\ell-1,\ell-1) - w(\ell,\ell)\} \MKQ g(Z_{\ell}) M_{\ell} -w(1,1) \MKQ g(Z_{1}) M_{1}\eqsp.
\end{align}
Note that $|\MKQ g(Z_{\ell})| \lesssim \taumix$, $|M_{\ell}| \lesssim \taumix$, and 
\[
w(\ell, \ell - 1) =
\begin{cases}
\frac{2(\ell - 1)}{b_n(n - b_n + 1)}, & \text{if } \ell < b_n, \\
\frac{2(b_n - 1)}{b_n(n - b_n + 1)}, & \text{if } b_n \le \ell \le n - b_n + 1, \\
\displaystyle \frac{2(n - \ell + 1)}{b_n(n - b_n + 1)}, & \text{if } \ell > n - b_n + 1\eqsp,
\end{cases}
\]
hence, with algebraic manipulations, we can check that 
\begin{equation}
\label{eq:differences_check_w_ell}
|w(\ell,\ell-1) - w(\ell-1,\ell-1) - w(\ell,\ell)| \leq \frac{2}{b_n (n-b_n+1)}\eqsp.
\end{equation}
Hence, with Minkowski's inequality, 
\begin{align}
\PE_{\xi}^{1/p}\bigl[ \bigl| S_1 \bigr|^{p}\bigr] \lesssim \frac{\taumix^2}{b_n}\eqsp.  
\end{align}
Similarly, for the term $S_2$, we write that 
\begin{align}
S_2 = \sum_{\ell=1}^{n} \bigl(\Delta^{(0,1)}(\ell,\ell) + w(\ell,\ell)\bigr) M_{\ell} \MKQ g(Z_{\ell-1}) = \sum_{\ell=1}^{n} \bigl( 2 w(\ell,\ell) - w(\ell,\ell-1)\bigr) M_{\ell} \MKQ g(Z_{\ell-1})\eqsp,
\end{align}
where in the last line we have additionally used the definition of coefficients $\Delta^{(0,1)}(\ell,\ell)$ from \eqref{eq:coef_def}. Hence, applying Minkowski's inequality and an upper bound $\bigl| 2 w(\ell,\ell) - w(\ell,\ell-1) \bigr| \lesssim \frac{1}{b_n (n-b_n+1)}$, we get that 
\begin{align}
\PE_{\xi}^{1/p}\bigl[ \bigl| S_2 \bigr|^{p}\bigr] \lesssim \frac{\taumix^2}{b_n}\eqsp.
\end{align}
Combining the above bounds concludes the proof.
\end{proof}

\begin{lemma}
\label{lem:bar_r_n_bound}
Under assumptions of \Cref{OBM concentration lemma}, it holds for any $p \geq 2$ that
\begin{equation}
\label{eq:bar_r_n_reminader_bound}
\PE_{\xi}^{1/p}|\RemU_{n}|^{p} \lesssim \frac{p \taumix^{2}}{\sqrt{b_n}} + \frac{p \taumix^2 \sqrt{b_n}} {n - b_n+1} + \frac{\taumix^2}{b_n}\eqsp.
\end{equation}
\end{lemma}
\begin{proof}
We estimate separately each of the terms from the error decomposition \eqref{eq:bar_r_n_representation}. First, note that for $\ell \geq 3$ and $1 \leq j \leq \ell - 2$, 
$$
|\Delta^{(1,0)}(\ell,j)| \leq \frac{2}{b_n(n-b_n+1)}.
$$ 
Hence, 
\begin{align}
\PE_{\xi}^{1/p}\bigl[\bigl| \MKQ g(Z_{\ell-1}) \sum_{j=1}^{\ell-2} \Delta^{(1,0)}(\ell,j) M_{j}\bigr|^{p}\bigr] 
&\lesssim \taumix \PE_{\xi}^{1/p}\bigl[\bigl| \sum_{j=1}^{\ell-2} \Delta^{(1,0)}(\ell,j) M_{j}\bigr|^{p}\bigr] \\
&\lesssim p \taumix \PE_{\xi}^{1/p}\bigl[\bigl(\sum_{j=1}^{\ell-2} \{\Delta^{(1,0)}(\ell,j)\}^2 M_{j}^2\bigr)^{p/2}\bigr] \\
&\lesssim \frac{p \taumix^{2}}{(n-b_n+1)\sqrt{b_n}}\eqsp. 
\end{align}
Thus, applying Minkowski's inequality, we get
\begin{align}
\PE_{\xi}^{1/p}\bigl[\bigl|\sum_{\ell=3}^{n} \MKQ g(Z_{\ell-1}) \sum_{j=1}^{\ell-2} \Delta^{(1,0)}(\ell,j) M_{j}\bigr|^{p}\bigr] \lesssim \frac{p \taumix^{2}}{\sqrt{b_n}}\eqsp.
\end{align}
Similarly, applying Burkholder's inequality, we obtain that 
\begin{align}
\PE_{\xi}^{1/p}\bigl[\bigl|\sum_{\ell=1}^{n} M_{\ell} \sum_{j=1}^{\ell} \Delta^{(0,1)}(\ell,j) \MKQ g(Z_{j-1})\bigr|\bigr] 
&\leq p \PE_{\xi}^{1/p}\biggl[\biggl( \sum_{\ell=1}^{n} \bigl(\sum_{j=1}^{\ell} \Delta^{(0,1)}(\ell,j) \MKQ g(Z_{j-1})\bigr)^2 M_{\ell}^2 \biggr)^{p/2}\biggr] \\
&\lesssim \frac{p \taumix^2}{\sqrt{n-b_n+1}}\eqsp.
\end{align}
Finally, applying Minkowski's inequality, we get
\begin{align}
\PE_{\pi}^{1/p}\bigl[\bigl|\sum_{\ell=3}^{n-1} \sum_{j=1}^{\ell-2} \Delta^{(1,1)}(\ell,j) \MKQ g(Z_{\ell-1}) \MKQ g(Z_{j-1})\bigr|^p\bigr] \lesssim \frac{\taumix^2}{b_n} \eqsp.
\end{align}
Applying now \Cref{lem:bar_r_n_remainder_bound}, we get that 
\begin{align}
\PE_{\pi}^{1/p}|\RemU_{n,rem}|^{p} \lesssim \frac{p \taumix^2 \sqrt{b_n}} {n - b_n+1} + \frac{\taumix^2}{b_n}\eqsp.
\end{align}
Similarly, applying \Cref{lem:bar_r_n_martigale_bound}, we get 
\begin{align}
\PE_{\pi}^{1/p}|\RemU_{n,mart}|^{p} \lesssim \frac{\taumix^2}{b_n}\eqsp.
\end{align}
Now it remains to combine above bounds in the decomposition \eqref{eq:bar_r_n_representation} and apply Minkowski's inequality once again.
\end{proof}

\section{Probability inequalities }
\label{appendix:Markov_technical}
We begin this section with a version of Rosenthal inequality \cite{Rosenthal1970}. Let $f: \Zset \to \rset$ be a bounded function with $\|f\|_{\infty} < \infty$. In this paper we need only a simplified version of the Rosenthal inequality, where the leading term with respect to number of terms $n$ in $\sum_{k=1}^{n}\{f(Z_k) - \pi(f)\}$ is governed not by the corresponding asymptotic variance $\sigma^2_{\infty}(f)$, but by the appropriate variance proxy. Below we prove the version of Rosenthal's inequality:
\begin{lemma}
\label{lem:auxiliary_rosenthal_weighted} 
Assume \Cref{assum:UGE}. Then for any $p \geq 2$ and $f: \Zset \rightarrow \rset$ with $\|f\|_{\infty} < \infty$, any initial distribution $\xi$ on $(\Zset,\Zsigma)$, and any coefficients $\beta_k \in \rset$, it holds that
\begin{equation}
\label{eq:rosenthal_weighted}
\begin{split}
\PE_{\xi}^{1/p}\bigl[\bigr| \sum_{k=1}^{n} \beta_k \{f(Z_k) - \pi(f)\} \bigr|^p\bigr] 
&\leq (16/3)\taumix p^{1/2}\|f\|_{\infty}(\sum_{k=2}^{n}\beta_k^2)^{1/2}\\
&\qquad +(8/3) \taumix \bigl( |\beta_1| + |\beta_{n}| + \sum_{k=1}^{n-1}|\beta_{k+1}-\beta_{k}|\bigr) \|f\|_{\infty}\eqsp.
\end{split}
\end{equation}
\end{lemma}
\begin{proof}
Under assumption \Cref{assum:UGE} the Poisson equation 
\[
g(z) - \MKQ g(z) = f(z) - \pi(f)
\]
has a unique solution for any bounded $f$ (see \cite[Chapter~21]{douc:moulines:priouret:soulier:2018}), which is given by the formula 
\[
g(z) = \sum_{k=0}^{\infty}\{\MKQ^{k}f(z) - \pi(f)\}\eqsp.
\]
Using \Cref{assum:UGE}, we obtain that $g(z)$ is also bounded with
\begin{align}
| g(z) | \leq \sum_{k=0}^{\infty} |\MKQ^{k}f(z) - \pi(f)| \leq 2\|f\|_{\infty} \sum_{k=0}^{\infty} (1/4)^{\lfloor k/\taumix \rfloor} \leq (8/3) \taumix \| f\|_{\infty}\eqsp.
\end{align}
Hence, we can represent
\begin{equation}
\label{eq:repr_weighted_stat}
\begin{split}
\sum_{k=1}^{n} \beta_k (f(Z_k) - \pi(f)) 
&= \underbrace{\sum_{k=2}^{n} \beta_k (g(Z_{k}) - \MKQ g(Z_{k-1}))}_{T_1} \\
&\quad \underbrace{+\sum_{k=1}^{n-1} (\beta_{k+1} - \beta_{k}) \MKQ g(Z_{k}) + \beta_{1} g(Z_1) - \beta_{n} \MKQ g(Z_{n})}_{T_2}\eqsp.
\end{split}
\end{equation}
The term $T_2$ can be controlled using Minkowski's inequality:
\begin{equation}
\label{eq:T_2_bound_rosenthal_remainder}
\PE_{\xi}^{1/p}[\bigl| T_{2}\bigr|^{p}] \leq (8/3) \taumix \bigl(|\beta_1| + |\beta_{n}| + \sum_{k=1}^{n-1}|\beta_{k+1}-\beta_{k}| \bigr) \|f\|_{\infty}\eqsp.
\end{equation}
Now we proceed with $T_1$. Set $\F_{k} = \sigma(Z_{\ell},\ell \leq k)$. Since $\CPE{g(Z_{k}) - \MKQ g(Z_{k-1})}{\F_{k-1}} = 0$ a.s., and $|g(Z_k) - \MKQ g(Z_{k-1})| \leq (16/3) \taumix\|f\|_{\infty}$, we get, using the Azuma-Hoeffding inequality \cite[Corollary 3.9]{vanHandel2016}, that 
\begin{equation}
\PP_{\xi}[|T_1| \geq t] \leq 2\exp\biggl\{-\frac{2t^2}{(16/3)^2\taumix^2\|f\|_{\infty}^2\sum_{k=2}^{n}\norm{\beta_k}^2}\biggr\}\eqsp.
\end{equation}
Hence, applying \Cref{lem:bound_subgaussian} we get
\begin{equation}
\PE^{1/p}_{\xi}[|T_1|^p] \leq (16/3) p^{1/2}\taumix \|f\|_{\infty}(\sum_{k=2}^{n}\norm{\beta_k}^2)^{1/2}\eqsp,
\end{equation}
and the statement follows.
\end{proof}
We conclude this section with a standard moment bounds for sub-Gaussian random variables. Proof of this result can be found in \cite[Lemma~7]{durmus2022finite}.
\begin{lemma}
\label{lem:bound_subgaussian}
Let $X$ be a random variable satisfying $\PP(|X| \geq t) \leq 2 \exp(-t^2/(2\sigma^2))$ for any $t \geq 0$ and some $\sigma^2 >0$. Then, for any $p \geq 2$, it holds that $ \PE[|X|^p] \leq 2 p^{p/2}\sigma^p$.
\end{lemma}

\newpage 
\bibliographystyle{chicago}
\bibliography{references}

\end{document}